\documentclass[10pt,leqno]{amsart}
\usepackage{indentfirst, amssymb,amsmath,amsthm}
\usepackage{graphicx}
\usepackage{epstopdf}
\newtheorem*{theo}{Theorem}

\newcommand{\be}{\begin{equation}}
\newcommand{\ee}{\end{equation}}
\newcommand{\beas}{\begin{eqnarray*}}
\newcommand{\eeas}{\end{eqnarray*}}
\newcommand{\bea}{\begin{eqnarray}}
\newcommand{\eea}{\end{eqnarray}}

\begin{document}
\title[Euler's Limit]{Euler's Limit---Revisited }
\date{}
\author[B. Chakraborty and S. Chakraborty]{Bikash Chakraborty and Sagar Chakraborty}
\date{}
\address{Department of Mathematics, Ramakrishna Mission Vivekananda
Centenary College, Rahara, West Bengal 700 118, India. }
\email{bikashchakraborty.math@yahoo.com, bikash@rkmvccrahara.org}
\address{Department of Mathematics, Jadavpur University, Kolkata-700032, India. }
\email{sagarchakraborty55@gmail.com}
\maketitle

Let $e_{n}:=\left(1+\frac{1}{n}\right)^{n}$ for $n \in \mathbb{N}$. It is well known that the sequence $(e_{n})$ is monotone increasing and bounded, hence it is convergent. The limit of this sequence is the famous Euler number $e$. Here we establish a generalization of this limit.
\begin{theo}
Let $\{ a_{n}\}$  and $\{ b_{n}\}$ be two sequences of positive real numbers such that  $a_{n}\to +\infty$ and $b_n$ satisfying the asymptotic formula $b_n\sim k\cdot a_{n}$, where $k>0$. Then
\begin{equation*}\label{1}
\lim\limits_{n\to\infty}\left(1+\frac{1}{a_{n}}\right)^{b_{n}}= e^{k}.
\end{equation*}
\end{theo}
\begin{proof} Let $f:(1, \infty)\to \mathbb{R}$ be defined by $f(x)=x-1-\ln x$. Since $f'(x)>0$ for $x\in (1,\infty)$, thus $f$ is increasing on $(1, \infty)$. Again, for the function  $g:(1, \infty)\to \mathbb{R}$ which is defined by $g(x)=\ln x-1+\frac{1}{x}$, $g'(x)>0$ for $x\in (1,\infty)$. Thus $g$ is also increasing on $(1, \infty)$. Hence
$$1-\frac{1}{x}<\ln x<x-1 ~~~~\text{for}~~~~x>1.$$
Since $a_{n}>0$, thus $1+\frac{1}{a_{n}}>1$. Thus using the above inequality, we have
\begin{eqnarray*}\label{1}
\frac{1}{1+a_{n}}<\ln(1+\frac{1}{a_{n}})<\frac{1}{a_{n}}.
\end{eqnarray*}
Since $b_{n}>0$, we have
\begin{eqnarray*}\label{1}
\frac{b_{n}}{1+a_{n}}<b_{n}\cdot\ln(1+\frac{1}{a_{n}})<\frac{b_{n}}{a_{n}}.
\end{eqnarray*}
Since $b_n\sim k\cdot a_{n}$, using Sandwich Lemma (\cite{BS}), we have $$ \lim\limits_{n\to\infty}\left(1+\frac{1}{a_{n}}\right)^{b_{n}}= e^{k}.$$
\end{proof}
It can be seen that by choosing $a_{n}=n$ and $b_{n}=n$, we get Euler's limit. Moreover, if $\frac{b_{n}}{a_{n}}\sim 0$, then $\lim\limits_{n\to\infty}\left(1+\frac{1}{a_{n}}\right)^{b_{n}}= 1$. Also, if $\frac{a_{n}}{b_{n}}\sim 0$, then $\lim\limits_{n\to\infty}\left(1+\frac{1}{a_{n}}\right)^{b_{n}}= \infty$.
\begin{center}
\textbf{Acknowledgements}
\end{center}
  We are thankful to the anonymous Referee and Prof. Gerry Leversha for their careful reading the manuscript and necessary suggestions.

\end{document}